\documentclass[11pt,a4paper]{article}
\usepackage{amssymb}
\usepackage{amsmath,amsfonts,amsthm}

\usepackage{graphicx}
\usepackage{amsmath}

\usepackage{amsfonts}
\usepackage{amsthm,amscd}
\usepackage{graphicx}
\usepackage{amsmath}
\usepackage{amssymb}
\usepackage{amstext}

\theoremstyle{plain}
\newtheorem{theorem}{Theorem}[section]
\newtheorem{corollary}[theorem]{Corollary}
\newtheorem{proposition}[theorem]{Proposition}

\theoremstyle{definition}
\newtheorem{definition}[theorem]{Definition}
\newtheorem{remark}[theorem]{Remark}
\newtheorem{example}[theorem]{Example}

\newcommand{\N}{\mathbb{N}}
\newcommand{\R}{\mathbb{R}}

\def\.{\cdot}

\def\R{{\mathbb R}}

\def\N{{\mathbb N}}
\def\Z{{\mathbb Z}}
\def\L{{\mathcal L}}
\def\A{{\mathcal A}}

\def\B{{\mathcal B}}

\def\L{{\mathcal L}}
\def\S^1{\mathbb{S}^1}
\def\M{{\mathcal M}}

\def\P{{\mathcal P}}

\def\S{{\mathcal S}}
\def\.{\cdot}

\def\({\left(}
\def\){\right)}

\begin{document}

\title{The involution kernel and the dual potential for functions in the Walters’ family}

\author{L. Y. Hataishi and  A. O. Lopes}

\maketitle

\centerline{Instituto de Matem\'atica e Estat\'istica, UFRGS - Porto Alegre, Brasil.}

\begin{abstract}

First, we set a suitable notation.  
Points in $\{0,1\}^{\mathbb{Z}-\{0\}} =\{0,1\}^\mathbb{N}\times \{0,1\}^\mathbb{N}=\Omega^{-} \times \Omega^{+}$,
are denoted by
 $( y|x) =(...,y_2,y_1|x_1,x_2,...)$, where $(x_1,x_2,...) \in \{0,1\}^\mathbb{N}$, and  $(y_1,y_2,...) \in \{0,1\}^\mathbb{N}$. The bijective map   $\hat{\sigma}(...,y_2,y_1|x_1,x_2,...)= (...,y_2,y_1,x_1|x_2,...)$ is called the bilateral shift and acts on $\{0,1\}^{\mathbb{Z}-\{0\}}$. 
Given $A:  \{0,1\}^\mathbb{N}=\Omega^+\to \mathbb{R}$ we express $A$ in the variable $x$, like $A(x)$.
 In a similar way,
 given $B:  \{0,1\}^\mathbb{N}=\Omega^{-}\to \mathbb{R}$ we express $B$ in the variable $y$, like $B(y)$.
 Finally, given $W: \Omega^{-} \times \Omega^{+}\to \mathbb{R}$, we express $W$
 in the variable $(y|x)$, like $W(y|x)$. By abuse of notation  we write $A(y|x)=A(x)$ and $B(y|x)=B(y).$ The probability $\mu_A$ denotes the equilibrium probability  for $A: \{0,1\}^\mathbb{N}\to \mathbb{R}$.
 
 Given a continuous potential $A: \Omega^+\to \mathbb{R}$, we say that the continuous  potential $A^*: \Omega^{-}\to \mathbb{R}$ is the dual potential of $A$, if there exists a continuous  $W: \Omega^{-} \times \Omega^{+}\to \mathbb{R}$, such that,  for all $(y|x) \in \{0,1\}^{\mathbb{Z}-\{0\}}$
$$
  A^* (y) =  \left[  A \circ \hat{\sigma}^{-1} + W \circ \hat{\sigma}^{-1} - W \right] (y|x).
 $$
 We say that $W$ is an involution kernel for $A$. The function  $W$ allows you
to define an spectral projection in the  linear space of the main eigenfunction of the Ruelle operator for $A$.
Denote by $\theta: \Omega^{-} \times \Omega^{+} \to
 \Omega^{-} \times \Omega^{+}$ the function
 $\theta(...,y_2,y_1|x_1,x_2,...)= (...,x_2,x_1|y_1,y_2,...).$
 We say that $A$ is symmetric if $A^* (\theta(x|y))= A(y|x)= A(x).$ To say that $A$ is symmetric is equivalent to saying that $\mu_A$ has zero entropy production.  Given $A$, we describe explicit expressions for  $W$ and the dual potential $A^*$, for $A$ in a family of functions  introduced by P. Walters. We present conditions for  $A$ to be symmetric and to be of twist type.


 
Email of Artur O. Lopes is arturoscar.lopes@gmail.com

Email of Lucas Y. Hataishi is lucas.hataishi@gmail.com

\end{abstract}

\maketitle

\section{Introduction}

The set $\mathbb{N}=\{1,2,3,...,n,..\}$ represents the one-dimensional  unilateral lattice. 
Denote $\Omega = \Omega^{+}= \{0,1\}^\mathbb{N}.$ The set $\{0,1\}$ is the set of symbols (or spins) of the symbolic space $\Omega.$ Points in $\Omega$ are denoted by $(x_1,x_2,x_3,...)$,  $x_j \in \{0,1\},$ $j \in \mathbb{N}.$  In some specific models in Statistical Mechanics the symbol $0$ can represent the spin $-$ and the symbol $1$ can represent the spin $+$.

The natural  metric on $\{0,1\}^\mathbb{N}$ is such that $d(x,y)=2^{-j}$, where $j\in \mathbb{N}$ is the first one such that $ x_j\neq y_j.$

  We denote by $ \Z^-_*$ the set such that $ \Z^-_* \cup \mathbb{N}= \{...,-3,-2,-1,1,2,3,...\}=\mathbb{Z}-\{0\},$ which represents  the one-dimensional bilateral lattice. 
  
Define $\Omega^- :=  \left\lbrace 0,1 \right\rbrace^{\Z^-_*}$, and   endow $\Omega^-$ with a metric space structure analogue to the  metric which was defined for $\Omega$.
	
%	 it gives the unique topology on $\Omega^-$ relatively to which the map $\theta : (x_{-n})_{n \in \N} \mapsto (x_n)_{n \in \N}$ is a homeomorphism (see also expression \eqref{teter}).
	
	 The cartesian product $\Omega^- \times \Omega=\Omega^- \times \Omega^{+}$ is denoted by $\hat{\Omega}$, and a general element described  by the ordered pair $((x_n)_{n \in \Z^-_*},(x_n)_{n \in \N} )$. We use the symbol $|$ for a better notation, and pairs can be  written as $((x_n)_{n \in \Z^-_*}|(x_n)_{n \in \N} )$.

It is also natural to identify $\Omega^{-}$ with $\Omega=\Omega^{+}$ and we will do this without mention.
$\hat{\Omega}$ is in some sense a version of $\{0,1\}^{\mathbb{Z}- \{0\}}\sim \{0,1\}^\mathbb{N} \times \{0,1\}^\mathbb{N}=\Omega \times \Omega.$  Under such point of view we prefer the following notation: given $(x_1,x_2,x_3,...)\in \Omega$ and $(y_1,y_2,y_3,...)\in \Omega$, then a general point in
$\hat{\Omega}$ is written as 
$$(y|x)=(...,y_3,y_2,y_1\,|\,x_1,x_2,x_3,...)\in \hat{\Omega}= \Omega^{-} \times \Omega^{+}.$$
	 
We denote by 
$\mathcal{ C}(\Omega)$  the set of continuous functions on  $\Omega$ taking real values. 
Cylinder sets in $\Omega$ are denoted by 
$[a_1,a_2,...,a_n]$, $a_j \in \{0,1\}$, $j=1,2,....,n.$
	 
The {\em natural extension} of a potential $A \in {\mathcal C}(\Omega)= {\mathcal C}(\Omega^+)$ to $\hat{\Omega}$ is the potential $\hat{A}: \hat{\Omega} \to \R$ given by
	\begin{align}
	\hat{A}(y|x) = A(x) \ , \ \forall \ (y|x) \in \hat{\Omega} \ .
	\label{chap1eq1}
	\end{align}

	When $B(y|x)$ does not depend on $y$ we  will  use sometimes the simplified  expression $B(y|x) = B(x).$

	The bijective map   
	$$(...,y_2,y_1|x_1,x_2,...) \to \hat{\sigma}(...,y_2,y_1|x_1,x_2,...)= (...,y_2,y_1,x_1|x_2,...)$$ 
	is called the bilateral shift and acts on $\{0,1\}^{\mathbb{Z}-\{0\}}$.
	
	The  map  $\sigma: \Omega  \to \Omega$, such that
	$$(x_1,x_2,...) \to \sigma (x_1,x_2,x_3,...)= (x_2,x_3,...)$$ 
	is called the unilateral shift and acts on $\Omega^+=\Omega$.

	Given $y=(y_1,y_2,y_3,..),$ we will also consider the function 
	$$x= (x_1,x_2,x_3,..)\to \tau_y(x) = (y_1,x_1,x_2,x_3,..),$$ 
	for $\tau_y:\Omega  \to \Omega$. The function $\tau_y$ is sometimes called an inverse branch.
	
	\begin{definition}
		Given  a continuous function $A :\Omega  \to \R$, if there exist  continuous functions $W: \hat{\Omega} \to \R$ and $A^* :\Omega^- \to \R$, such that, for any $(y|x) \in \Omega$, 
		
		\begin{equation}  \label{treo}
		A^* (y) =  \left[  \hat{A} \circ \hat{\sigma}^{-1} + W \circ \hat{\sigma}^{-1} - W \right] (y|x),
		\end{equation}
then we say that $W$ is an {\bf involution kernel} for $A$, and that $A^*$ is the {\bf  dual potential} of $A$ relatively to $W$.
		\label{chap1defi2}
	\end{definition}

	When $B(y|x)$ does no depend on $x$ we sometimes  will  use the simplified  expression $B(y|x) = B(y)$.

Given $A: \Omega^+ \to \mathbb{R}$, the involution kernel $W$ is not unique. One can show that in the case $A$ is H\"older there exist $W$ and a H\"older function $A^*: \Omega^- \to \mathbb{R}$ satisfying \eqref{treo} (see \cite{BLT}, \cite{CLO}, \cite{LMMS}, \cite{SV}, \cite{LV} and \cite{VV1}). We will consider here  a large class of functions such that some of them are not of of  H\"older class.

 Given $A$ we are interested in explicit expressions for the involution kernel $W$ and for the dual potential $A^*.$

The involution kernel was introduced in \cite{BLT} where it was shown that 
a natural way to obtain an involution kernel $W$ for $A$ is via expression \eqref{chap1eq2} (or \eqref{Wkernel}).

Note that in the case $W$ is an involution kernel for $A$, then, given $\beta \in \mathbb{R}$, we get that $\beta W$ is an involution kernel for $\beta A$.  The value $\beta$ corresponds in Thermodynamic Formalism (and Statistical Mechanics) to $\frac{1}{T}$, where $T$ is temperature,
and $A$ corresponds  to minus the Hamiltonian.
 In the case $A$ is H\"older, using this property  and \eqref{tritri},  large deviation properties when the temperature goes to zero are obtained in \cite{BLT}.
Questions related to the selection of probability (and subaction) when the temperature goes to zero appear in \cite{BLL} and \cite{BLM}.

 Denote by $\theta: \Omega^{-} \times \Omega^{+} \to
 \Omega^{-} \times \Omega^{+}$ the function such that 
 \begin{equation} \label{teter}\theta(...,y_2,y_1|x_1,x_2,...)= (...,x_2,x_1|y_1,y_2,...).
 \end{equation}
 
 We say that $A$ is {\bf symmetric} if $A^* (\theta(x|y))= A(y|x)= A(x).$

Note that considering the set of symbols $\{-1, 1\}  $  instead of $\{0,1\}$ does not change much the above definitions.
		
\begin{example} \label{ezze} (Taken from  Section 5 in \cite{CL}) Consider the alphabet $\{-1,1\}$ and the symbolic space  $\{-1,1\}^\mathbb{N}$. In this case $\hat{\Omega}$ is $\{-1,1\}^{\mathbb{Z}-{0}}$. We will define a potential  $A:\{-1,1\}^\mathbb{N}\to \mathbb{R}$ which is symmetric. Indeed, consider a sequence $a_n>0$, $n \geq 1$, such that $\sum_{i\geq 1}\sum_{j>i} a_j <\infty$, and for $x=(x_1,x_2,x_3,...) \in \{-1,1\}^\mathbb{N}$, we define $A(x) =\sum_{n=1}^\infty a_n x_n$ (it is called  a product type potential). Consider $W: \{-1,1\}^{\mathbb{N}\times \mathbb{N}} \to \mathbb{R}$, given by
	$$ W(y|x)=\sum_{i=1}^\infty [(x_i + y_i) (\sum_{j>i}a_{j})] =\sum_{i=1}^\infty (x_i + y_i) (a_{i+1} + a_{i+2}+....).$$
	Then, one can show that $W$ is an involution kernel and $A$ is symmetric.
	A particular example is when $a_n=2^{-n},$ $n \geq 1$, in which case $A$ is of H\"older class.

\end{example}

It is known that
 eigenfunctions for the Ruelle operator for a H\"older potential  $A: \Omega^{+} \to \mathbb{R}$ and eigenprobabilities for the dual of the Ruelle operator for $A^*:\Omega^{-} \to \mathbb{R}$ are related via the involution kernel (see \cite{BLT} or \cite{LMMS}), which plays the role of a dynamical integral kernel (see expression \eqref{tritri}).

We elaborate on that:  $ \mathcal{L}_A $ denotes the Ruelle operator for $A: \Omega^{+} = \Omega\to \mathbb{R}$, which is the linear operator acting on continuous function $f_0 :\Omega^{+} \to \mathbb{R}  $, such that,
$$f_0 \to f_1(z) =   \mathcal{L}_A (f_0) (z) = \sum_{ \sigma(x)=z} e^{A (x)}  f_0 (x).$$    

When looking for equilibrium probabilities on $\hat{\Omega}$ (which are  invariant for $\hat{\sigma}$) for potentials $\hat{A}:\hat{\Omega} \to \mathbb{R}$ properties of the Ruelle operator, as defined above, are quite useful (see \cite{PP} or Appendix in \cite{L3}).

Given $A^*$, the Ruelle operator  $ \mathcal{L}_{A^*} $ acts on functions $f_0 :\Omega^{-} \to \mathbb{R}$.

The Ruelle theorem for an H\"older function  $A$ claims the existence of a positive eigenfunction $\varphi_A$ for  the operator  $\mathcal{L}_A$ and an  associated to the eigenvalue $\lambda_A>0$. (see \cite{PP} for a proof).  $\varphi_A$ is called the main eigenfunction of $\mathcal{L}_A.$
The dual operator for $A$ is denoted $\mathcal{L}_A^*$ and acts on finite measures. The Ruelle theorem helps to identify the equilibrium probability for the potential $A$ (see \cite{PP}).

Given a  H\"older potential $A$, we say that the probability   $\nu_A$ on $\Omega=\Omega^+$  is the eigenprobability for the dual of the Ruelle operator  $\mathcal{L}_A^*$, if
 $\mathcal{L}_A^*(\nu_A) =\lambda_A\, \nu_A$. Note that from \cite{BLT} (or \cite{LMMS}) the main eigenvalue of the Ruelle operator for $A$ and the  main eigenvalue of the Ruelle operator for its dual $A^{*}$ are the same. The same thing for $\mathcal{L}_A^*$ and $\mathcal{L}_{A^{*}}^*.$ That is $\lambda_A= \lambda_{A^*}$. We denote by $\nu_{A^*}$ the eigenprobability for the dual  operator $\mathcal{L}_{A^{*}}^*.$ The probability $\nu_{A^*}$ is defined on $\Omega^{-}=\Omega.$

Denote by $W=W_A$ the involution for $A$, then,
\begin{equation} \label{tritri} \int e^{ W (y|x)} d \nu_{A^{*}}(y) = \varphi_A(x)
\end{equation}
is the main eigenfunction of the Ruelle operator $\mathcal{L}_A$ (see, for instance \cite{LMMS}, for a proof).  The involution kernel  allows you
to define an spectral projection in the linear space of the main eigenfunction  of the Ruelle operator $\mathcal{L}_A$ (see page 482 in \cite{LOS}).

Note that in equation  \eqref{treo} for the involution kernel the eigenvalue $\lambda_A$ does not appear; this is a property  which is eventually useful when trying to get the main eigenfunction in the study of a particular example of potential $A$.

An interesting fact is that for a general continuous potential $A:\Omega \to \mathbb{R}$ an eigenprobability always exists but a continuous positive eigenfunction not always (see \cite{CDLS}, \cite{CSS} and \cite{CLS}).

Knowing the involution kernel for the potential $A$,  other eigenfunctions for $\mathcal{L}_A$ (not strictly positive) can be eventually obtained via 
eigendistributions for $\mathcal{L}_{A^*}^*$ (see \cite{GLP}).

When $A$ is symmetric we get 
\begin{equation} \label{tritritu} \int e^{ W (y|x)} d \nu_{A}(y) = \varphi_A(x).
\end{equation}
\smallskip

The invariant probability $\mu_A = \varphi_A\, \nu_A$ maximizes topological pressure and is called the equilibrium probability for $A$ (see \cite{PP}). In the case, there exists the limit $\mu_{\beta \,A}$, when $\beta \to \infty$, in the weak-$*$ convergence sense,  we say there exists a selection of probability at temperature zero. 

In \cite{BLT}, using properties of the involution kernel  the authors proved Large Deviation Theorems for the zero-temperature limit in the case there exist selection (\cite{Mengue} considers the case where  it is not assumed  selection). In another direction, in \cite{LM} it is shown that the involution kernel appears as a natural tool for the investigation of entropy production of  $\mu_A$.  If the potential $A$ is symmetric then the entropy production of $\mu_A$ is zero (see  Section 7 in \cite{LM}). General results for entropy production in Thermodynamic Formalism appears in \cite{PoSh}, \cite{Rue}, \cite{Rue1}, \cite{Galla}, \cite{Maes} and \cite{Jiang}.
The case of symbolic spaces where the alphabet is a compact metric space is considered in  \cite{LM}.  

  In Example 2  (product type potentials) and Example 3 (Ising type potentials) in Section 5 in \cite{CL} the authors present examples of potentials  $A$ in $\{-1,1\}^\mathbb{N}$ that are symmetric by exhibiting the explicit expression of the involution kernel (see our Example \ref{ezze}). From this property and results in \cite{LM} it follows that the equilibrium probabilities $\mu_A$ for such potentials have zero entropy production. Explicit expressions for the involution kernel (and the dual potential) where obtained in several sources as in Proposition 9 in \cite{BLT}, Section 5 in \cite{BLLCo},  in Remarks 6 and 7 in \cite{LOT} and in Section 13.2 in \cite{FLO}.
  
	In the works  \cite{LOT}, \cite{CLO} and \cite{LOS}, the involution kernel  $W: \Omega^- \times \Omega \to \mathbb{R}$ is considered as a dynamical cost in an   Ergodic Optimal cost Problem. In the classical setting, a lot of nice results on Optimal Transport are proved under the condition of  convexity of the cost function.   One can define  the Twist Condition (see Definition \ref{kil}) for an involution kernel $W$ (of a potential $A$) and this plays the role of a form of convexity in Ergodic Transport. We will address this issue in Section \ref{twi}. The twist condition is sometimes called the supermodular condition
(see section 5.2 in \cite{Mit}), which a natural hypothesis in optimization problems (see \cite{Bha}), and Aubry-Mather theory (see \cite{CLO}, \cite{LOS} and \cite{LOT}).

Here, we investigate the existence of the involution kernel $W: \Omega^- \times \Omega \to \mathbb{R}$ and we  study  duality and symmetry issues for potentials $g:\{0,1\}^\mathbb{N}\to \mathbb{R}$ in  a certain class of continuous potentials $g$ (see Sections 
\ref{IIn} and \ref{secdualpotential}) to be defined next. We present explicit expressions.

We will consider a potential $g:\{0,1\}^\mathbb{N}\to  \mathbb{R}$ which is at least continuous.
Potentials on the so-called Walters family were introduced in \cite{Wal2} where some explicit results for the main eigenfunction of the Ruelle operator were obtained. For this family of potentials in \cite{CHLS} the authors  study  a certain class of Spectral Triples, and in \cite{BLM} the authors present explicit expressions for subactions in Ergodic Optimization.

\begin{definition} \label{ouy}
A potential $g: \Omega \to \mathbb{R}$ is in the Walters family, if and only if, there exists convergent sequences
$(a_n)_{n \in \mathbb{N}}$, $(b_n)_{n \in \mathbb{N}}$, $(c_n)_{n \in \mathbb{N}}$ and $(d_n)_{n \in \mathbb{N}}$, such that, for any $x \in \Omega$,
\begin{itemize}
\item $g(0^{n+1}1x) = a_{n+1}$ ;
\item $g(01^n0x) = b_n$ ;
\item $g(1^{n+1}0x) = c_{n+1}$ ;
\item $g(10^n1x) = d_n$,
\end{itemize}
for all $n \in \mathbb{N}$.
\label{defi1}

The set of such continuous potentials is described by $R(\Omega).$ 
\end{definition}

Some of these functions $g$ are  not of  H\"older class.

We denote by $a, b, c, d$, respectively, the limits of the sequences  $a_n, b_n, c_n, d_n.$

Theorem 1.1  in \cite{Wal2} describes the condition for the potential $g$ (on the Walters family) to be in the Bowen's class or the Walters' class (do not confuse this set with the class of Walters potentials described by  Definition \ref{ouy}).  Theorem 3.1 in \cite{Wal2} presents conditions on $g$ which will imply that the Ruelle Theorem is true (see \cite{PP}). 

If
\begin{equation} \label{HolW} a_n-a,\, b_n-b,\, c_n-c\, \,\text{and} \,\,d_n-d
\end{equation}
converge to zero exponentially fast to zero, then, the potential $g$ is of 
 H\"older class. In this case the equilibrium states are unique (no phase transition) and the pressure function is differentiable.

The family described by Definition \ref{ouy}  contains a subfamily of potentials called of the Hofbauer potentials  (see \cite{Hof}, \cite{Lo1}, \cite{Lofi}, \cite{FL}, \cite{LopR} and \cite{CL1}) which are not of H\"older class.  For this subclass,  in some cases, there exists more than one equilibrium state (phase transition happens) and the pressure function may not be differentiable. The Hofbauer example corresponds to the following case: take $1<\gamma$, then
consider $c_n=- \gamma\, \log (\frac{n+1}{n})$, $ n \geq 2,$ $d_n = - \gamma \log (2),$ and $a_{n+1} = d_n = \log (\zeta(\gamma)),$ where 
$\zeta(\gamma)$ is the Riemman zeta function. Questions related to renormalization for this class of potentials appear in \cite{BLL1} and \cite{LopR}; the Hofbauer potential is a fixed point for the renormalization operator.

The variety of possibilities of the functions on   the Walters' family of potentials is so rich that given a  certain sequence 
$\mathfrak{c}_n$, $n \in \mathbb{N},$
describing a possible decay of correlations (under some mild assumptions), one can find a potential in the family such the equilibrium probability for this potential  has  this decay of correlation (see \cite{LopR}).

	The potential of Example \ref{ezze} is not in the Walters' family.

\medskip

%\section{Introduction}
	
		A positive continuous function  $g : \Omega \to (0,1)$, satisfying
	$ g(0x) + g(1x) = 1$, for all  $x \in \Omega$, is called a 
	$g$-function.
Let $G(\Omega)$ denote the set of all g-functions on the Walters family.
		Corollary 2.3 in \cite{Wal2} (see also our Corollary \ref{veve}) implies  that,  the equilibrium probability $\mu_A$ of the continuous potential $A=\log g$, where $g\in G(\Omega)$ is  on the Walters family,
	satisfies $\forall n \in \mathbb{N}$, 
	and each pair of cylinders $[x_1,x_2,...,x_{n-1},x_n]$ and $[x_n,x_{n-1},...,x_2,x_1]$
	\begin{equation} \label{hahas} \mu_A ([x_1,x_2,...,x_{n-1},x_n])= \mu_A ([x_n,x_{n-1},...,x_2,x_1]).
	\end{equation}
		
		This symmetry on the measure of cylinders means zero entropy production (see \cite{Jiang} and Section \ref{lala}) and is related to the concept of $A$ being symmetric via the involution kernel (see  \cite{LM}). It follows from the reasoning of  \cite{LM}  (using \cite{Jiang} and \eqref{hahas}) that  given a H\"older potential $A$ on the Walters family, there exists an involution kernel $W$ such that makes $A$ symmetric.
	
	Given $x' = (x_1',x_2',...,x_n' ,...)\in \Omega$ fixed, if 
	\begin{align}
	W(y|x)= \sum_{n \in \N} \left[ \hat{A} \circ \hat{\sigma}^{-n} (y|x) - \hat{A} \circ \hat{\sigma}^{-n} (y|x') \right]
	\label{chap1eq2}
	\end{align}
	is convergent for any $(y|x) \in \hat{\Omega}$, it is shown in Section \ref{secdualpotential} that $W$ is an involution kernel for $A$ (see also \cite{LMMS}). It is called the {\bf involution kernel} of $A$ based on $x' \in \Omega$.

	Given a continuous potential $A:\Omega \to \mathbb{R}$, a  more simple expression   for  the involution kernel $W$ for $A$ is
	
	\begin{equation}\label{Wkernel} W(y|x)=\sum_{n\geq 1}A(y_n,...,y_1,x_1,x_2,...)-A(y_n,...,y_1,x_1',x_2',...), \end{equation}
	
	where $x=(x_1,x_2,...,x_n,..)$ and $y=(y_1,y_2,...,y_n,..)$.

	In the case \eqref{Wkernel} converges an expression for a dual potential $A^*$ associated to such involution kernel is given by \eqref{elew}.

	If $A$ is of H\"older class the above sum   \eqref{Wkernel} always converges.
	
	\medskip
	
	An outline of our main results: we investigate the existence and explicit expressions for $W$ (see Section \ref{IIn}) and  the dual potential $A^*$ (see   Tables \eqref{dualpotentialtable1}, \eqref{dualpotentialtable2} and \eqref{dualpotentialtable3} in Subsection \ref{kyu1}), for potentials $A$ in the class $R(\Omega)$.  Note that we just require convergence and not much regularity (like  H\"older continuity). We give a sufficient and necessary condition for the existence of $W$ (see Theorem \ref{okl}). For results about symmetry of the potential $A$ see expression \eqref{chap1sec3subs1eq7} and  Theorem \ref{puxa1}.

	For a potential $g$ in the Walters' family
we present  here  in Section \ref{twi} conditions for a potential of such type to satisfy the relaxed
	twist condition (see Definition \ref{kil} and Theorem \ref{kjg}).
	
Questions related to the characterization of normalized Potentials on $G(\Omega)$ are described in Section \ref{nono} (see expression \eqref{normpotentialeq6}).

	In Section \ref{lala} we consider the sets
	$\M(\sigma)$ and $\M(\hat{\sigma}),$ which are, respectively, the set of Borel invariant probability measures for $\sigma: \Omega \to \Omega$ and $\hat{\sigma}:\hat{\Omega} \to \hat{\Omega}$. The reasoning we followed above deals with the Equilibrium Statistical Mechanics of the
	lattice $\mathbb{N}$, which results in probabilities on $\mu_A\in\M(\sigma)$. On the other hand, the study of equilibrium probabilities on the lattice $\mathbb{Z}-\{0\}$ result in probabilities $\mu_{\hat{A}}$ on $\M(\hat{\sigma}),$ where $\hat{A}: \hat{\Omega}=\{0,1\}^{\mathbb{Z}-\{0\}} \to \mathbb{R}.$ The relation of this two frameworks is the topic of Section
	\ref{lala} (see also appendix  in \cite{L3}).

	\section{Ruelle's Theorem}
	
	We briefly describe some properties of  the  Ruelle's Operator Theorem before addressing the main issues of our paper.
	
	\begin{definition}
		A continuous function $A: \Omega \to \R$ is an element of $\Xi(\Omega)$ if, and only if, for every $\epsilon > 0$, there exists $\delta > 0$ such that, $\forall \ n \in \N$ and $\forall \ x \in \Omega$, $y \in B(x,n,\delta)$ implies
		\begin{align*}
		\left| \sum_{j=0}^{n-1} A(\sigma^j(x)) - \sum_{j=0}^{n-1} A(\sigma^j(y)) \right| < \epsilon \ ,
		\end{align*}
		\label{chap1sec2defi1}
		The set $\Xi(\Omega)$ is the set of potentials with Walters Regularity.
	\end{definition}
	\smallskip
	
	The Hofbauer potential does not have Walters Regularity.
	
	\smallskip
	
	In \cite{Wal2} is proved the following result:
	\begin{theorem}
		$A \in R(\Omega) \cap \Xi(\Omega) \iff \sum_{n \in \N} (a_n - a)$ and $\sum_{n \in \N} (c_n - c)$ are both convergent. If $A \in R(\Omega) \cap \Xi(\Omega)$ then Ruelle's Operator Theorem holds for $A$ and it has a unique equilibrium state.
	\end{theorem}
	
	It follows from our calculations the following result:
	
	\begin{theorem}
		Let $A \in \R(\Omega)$. If, for some $x' = (x'_n)_{n \in \N} \in \Omega$, the series
		\begin{align*}
		\left[ A(y_1 x) - A(y_1 x') \right] + \sum_{n = 2}^\infty \left[ A(y_n...y_1x) - A(y_n...y_1x') \right]
		\end{align*}
		converges absolutely, then it converges absolutely for any $x' \in \Omega$. Also, the convergence of the above series implies $A \in \Xi(\Omega)$, and $A$ has a unique equilibrium state.
		\label{teo2}
	\end{theorem}
	
	This result is related to expressions \eqref{Wkernel} and \eqref{poii}.

	\section{Existence of the Involution kernel}
	
	We are going to calculate   formally in Subsection \ref{IIn} the exact expression for the involution kernel (based on a certain point $x'$) of a given potential in $R(\Omega)$. The result also gives a sufficient and necessary condition for the existence of the involution kernel based on a certain point $x'$. We analyze this problem considering Definition \ref{chap1defi2}.
	
	After the calculations in Section \ref{IIn}, the final result is:
	
	\begin{theorem} \label{okl}
	For the involution kernel as in \eqref{chap1eq2} to exist, it is necessary and sufficient that the potential is of Walters regularity.
	\end{theorem}

	We investigate the existence of involution kernels of the form \eqref{chap1eq2} for the potentials in $R(\Omega)$. We give a sufficient and necessary condition for its existence and calculate, when this condition is satisfied, the exact expression of the involution kernel.
	\begin{definition}
	We call  $W$ be the involution kernel of $A$ based at the point $x' \in \Omega$, the function such that, for any ordered pair $(y|x) \in \hat{\Omega}$,

\begin{equation} \label{poii} W(y|x)=\, \sum_{n=1}^\infty \left[ A \circ \hat{\sigma}^{-n}(y|x) - A \circ \hat{\sigma}^{-n}(y|x') \right],
\end{equation}
in the case the sum converges.
\end{definition}

See \eqref{chap1sec1subsec1eeq3},  \eqref{chap1sec1subsec1eeq11} and \eqref{oiay} for affirmative cases with explicit expressions.

	\begin{definition}
		If $W$ is an involution kernel for $A \in R(\Omega)$, based at $x'$, we define the dual potential $A^*: \Omega \to \R$ to be the potential given by
		\begin{align}
		A^*(y) := A^*(y|x) = A(\tau_y(x)) + W \circ \hat{\sigma}^{-1} (y|x) - W(y|x) \ ,
		\label{poio}
		\end{align}
		(in the case the sum converges)
		which one can show does not depend on  $x$ (see \eqref{opiu} and \eqref{opiu1}).
		\label{chap1sec3defi1}
	\end{definition}

	Conditions for symmetry for the Walters' family  are given by \eqref{chap1sec3subs1eq7}.
	
	One can show that when $A$ is of  H\"older class, then expressions \eqref{poii} and  \eqref{poio} are well defined and they are, respectively, an involution kernel and a dual potential for $A$ (see for instance \cite{LMMS}).

	\section{Calculation of the Involution Kernel for the Walters' family} \label{IIn}
	
	We calculate formally an exact expression for of the involution kernel of a potential in $R(\Omega)$ via expression \eqref{poii}. The result gives a sufficient and necessary condition for the existence of the involution kernel  on the base point $x'$. We analyze different choices of $x'$.
	
	Let $A \in R(\Omega)$ be defined by convergent sequences of real numbers
	$$(a_n)_{n \in \N},(b_n)_{n \in \N},(c_n)_{n \in \N}\,\text{and}\,(d_n)_{n \in \N}.$$
	
	\subsection{First Case: $x' = 0^\infty$}
	
	With this choice of base point, $x = 0^\infty \implies W(\. | x) \equiv 0$. For $y = 0^\infty$ and $x \in [0^k1]$, with $k \in \N$, 
	\begin{align*}
	A(y_n...y_1x) - A(y_n...1x') 
	= A(0^{k+n}1...) - A(0^\infty)
	= a_{k+n} - a
	\end{align*}  
	\begin{align}
	W(0^\infty | [0^k1]) = \sum_{n \in \N} (a_{n+k} - a) \ .
	\label{chap1sec1subsec1eeq1}
	\end{align}
	If $y = 0^\infty$ and $x = 1^\infty$, then, since $A(01^\infty) - A(0^\infty)$ and $A(0^n1^\infty) - A(0^\infty) - a_n - a$ for natural numbers $n > 1$,
	\begin{align}
	W(0^\infty|1^\infty) = b-a + \sum_{n=2}^\infty (a_n - a) \ .
	\label{chap1sec1subsec1eeq2}
	\end{align}
	If $y = 0^\infty$ and $x \in [1^k0]$, $k \in \N$, then
	\begin{align}
	W(y|x) = b_k - a + \sum_{n=2}^\infty ( a_n - a) \ ,
	\label{chap1sec1subsec1eeq3}
	\end{align}
	since $A(01^k0...) - A(0^\infty) = b_k - a$ and $A(0^n 1^k 0...) - A(0^\infty) = a_n - a$ if $n > 1$.
	We obtained the expressions for the involution kernel when the first coordinate is the point $y = 0^\infty$.
	
	Consider now $y \in [0^l 1]$, with $l \in \N$. If $x \in [0^k 1]$. If $1 \leq n \leq l$, 
	\begin{align*}
	A(y_n...y_1x) - A(y_n...y_1x') = A(0^{n+k}1...) - A(0^\infty) = a_{n+k} - a \ .
	\end{align*} 
	Now
	\begin{align*}
	A(y_{l+1}...y_x) - A(y_{l+1}...y_1x') = A(10^{l+k}1...) - A(1 0 ^\infty) = d_{l+k} - d \ .
	\end{align*}
	The next terms are null, so
	\begin{align}
	W([0^l1]|[0^k1]) = \sum_{n=1}^l (a_{n+k} - a) + d_{l+k} - d \ .
	\label{chap1sec1subsec1eeq4}
	\end{align}
	
	Still for $y \in [0^l1]$, let $x = 1^\infty$. In this case, $A(y_1x) - A(y_1x') = A(01^\infty) - A(0^\infty) = b-a$. If $1 < n \leq l$, 
	\begin{align*}
	A(y_n...y_1x) - A(y_n...y_1x') = A(0^n 1^\infty) - A(0^\infty) = a_n - a \ .
	\end{align*}
	\begin{align*}
	A(y_{l+1}...y_1x) - A(y_{l+1}...y_1x') = A(10^l 1^\infty) - A(1 0^\infty) = d_l - d \ ,
	\end{align*}
	and all the other therm are null. Therefore,
	\begin{align}
	W([0^l1]| 1^\infty) = (b-a) + \sum_{n=2}^l (a_n - a) + d_l - d \ .
	\label{chap1sec1subsec1eeq5}
	\end{align}
	
	If $y \in [0^l1]$ and $x \in [1^k0]$, then
	\begin{align*}
	A(y_1x) - A(y_1x') = A(y_1x') = A(01^k0...) - A(0^\infty) = b_k - a
	\end{align*}
	For $1 < n \leq l$,
	\begin{align*}
	A(y_n...y_1x) - A(y_n...y_1) = A(0^n1^k0...) - A(0^\infty) = a_n - a \ .
	\end{align*}
	\begin{align*}
	A(y_{l+1}...y_1x) - A(y_{l+1}...y_1x') = A(10^{l}1^k0...) - A(10^\infty) = d_l -d \ .
	\end{align*}
	Since the other terms are zero, we conclude
	\begin{align}
	W([0^l 1]|[1^k0]) = b_k - a + \sum_{n=2}^l (a_n - a) + d_l - d \ .
	\label{chap1sec1subsec1eeq6}
	\end{align}
	
	The above equation closes the case $y \in [0]$. We proceed in a similar way to calculate it for $y \in [1]$.
	
	If $y = 1^\infty$ and $x \in [0^k 1]$, with $k \in \N$, then
	\begin{align*}
	A(y_n...y_1 x) - A(y_n...y_1x') = A(1^n 0^k 1...)-A(1^n 0^\infty) = \begin{cases}
	d_k - d,\ \text{if $n=1$} \\
	0,\ \text{if $n > 1$} 
	\end{cases}\ .
	\end{align*}
	
	First, consider $y = x = 1^\infty$.
	\begin{align*}
	A(y_n...y_1x) - A(y_n...y_1x') = A(1^\infty) - A(1^n 0^\infty) = \begin{cases}
	c - d \ , \text{if $n = 1$} \\
	c - c_n \ , \text{if $n > 1$} 
	\end{cases} \ ,
	\end{align*}
	which implies
	\begin{align}
	W(1^\infty | 1^\infty) = c-d + \sum_{n=2}^{\infty} (c - c_n) 
	\label{chap1sec1subsec1eeq7}
	\end{align}
	
	If $y = 1^\infty$ and $x \in [1^k0]$, then
	\begin{align*}
	A(y_1x) - A(y_1x') = A(1^{1+k} 0...) - A(10^\infty) = c_{k+1} - d \ ,
	\end{align*}
	and 
	\begin{align*}
	A(y_n...y_1x) - A(y_n...y_1x') = A(1^{n+k}0...) - A(1^n0^\infty) = c_{n+k} - c_n \ ,
	\end{align*}
	for $n > 1$. We have, therefore,
	\begin{align}
	W(1^\infty|[1^k0]) = c_{k+1} - d + \sum_{n=2}^\infty (c_{n+k}  - c_n) \ .
	\label{chap1sec1subsec1eeq8}
	\end{align}
	
	Consider now the case $y \in [1^l0]$ and $x \in [0^k1]$, for $l,k \in \N$.
	\begin{align*}
	A(y_1x) - A(y_1x') = A(10^k1...) - A(10^\infty) = d_k - d \ .
	\end{align*}
	If $1 < n \leq l$, then
	\begin{align*}
	A(y_n...y_1x) - A(y_n...y_1x') =A(1^n0^k1...) - A(1^n 0^\infty)  =c_n - c_n = 0 \ ,
	\end{align*}
	so
	\begin{align}
	W([1^l0]|[0^k1]) = d_k - d \ .
	\label{chap1sec1subsec1eeq9}
	\end{align}
	
	If $y \in [1^l0]$ and $x = 1^\infty$, 
	\begin{align*}
	A(y_1x) - A(y_1x') = A(1^\infty) - A(10^\infty) = c - d \ .
	\end{align*}
	For $1 < n \leq l$, then 
	\begin{align*}
	A(y_n...y_1x) - A(y_n...y_1x') = A(1^\infty) - A(1^n 0^\infty)  = c - c_n \ .
	\end{align*}
	\begin{align*}
	A(y_{l+1}...y_1x) - A(y_{l+1}...y_1x') = A(01^\infty) - A(01^l0^\infty) = b-b_l \ .
	\end{align*}
	The other terms are zero, so
	\begin{align}
	W([1^l0]|[0^k1]) = c-d + \sum_{n=2}^l (c - c_n) + b-b_l \ .
	\label{chap1sec1subsec1eeq10}
	\end{align}
	
	Let $y \in [1^l0]$ and $x \in [1^k0]$. 
	\begin{align*}
	A(y_1x) - A(y_1x') = A(1^{k+1}0...) - A(10^\infty) = c_{k+1} - d \ .
	\end{align*}
	For $1 < n \leq l$,
	\begin{align*}
	A(y_n...y_1x) - A(y_n...y_1x') = A(1^{n+k}0...) - A(1^n0^\infty) = c_{n+k} - c_n \ . 
	\end{align*}
	\begin{align*}
	A(y_{l+1}...y_1x) - A(y_{l+1}...y_1x') = A(01^{l+k}0...) - A(01^l0^\infty) = b_{l+k} - b_l \ .
	\end{align*}
	The other terms are zero, so
	\begin{align}
	W([1^l0]|[1^k0]) = c_{k+1} - d + \sum_{n=2}^{l}(c_{n+k} - c_n) + b_{l+k} - b_l
	\label{chap1sec1subsec1eeq11}
	\end{align}
	
	Equations \eqref{chap1sec1subsec1eeq1} - \eqref{chap1sec1subsec1eeq11} gives the involution kernel, and the convergence conditions are needed only in equations \eqref{chap1sec1subsec1eeq1}, \eqref{chap1sec1subsec1eeq2}, \eqref{chap1sec1subsec1eeq3}, \eqref{chap1sec1subsec1eeq8}. They are
	\begin{align}
	\exists \sum_{n \in \N} |a_n - a| \ , \ \exists \sum_{n \in \N} |c_n  - c| \ , \ \exists \sum_{n \in \N}|c_{n+k} - c_n| \ \forall \ k \in \N \ .
	\end{align}
	Observe, however, that the third convergence condition is implied by the second, since $|c_{n+k} - c_n| \leq |c_{n+k} - c| + |c - c_n|$. 
	
	All the calculations were made with the involution kernel based at $x' = 0^\infty$. Thus, it is important to investigate if the convergence hypothesis change if other point $x'$ is fixed as a base point. It is sufficient to consider $x' = [0^\alpha1]$ for some natural number $\alpha$. The other cases are obtained from these two by permutation of symbols $a \leftrightarrow c$ and $b \leftrightarrow d$.
	
	\subsection{Second case: $x' \in [0^\alpha1]$}
	
	We present below the calculations of the involution kernel based at a point $x' \in [0^\alpha1]$.   
	
	\begin{align*}
	W(10^l|0^k1) &= \sum_{j=1}^l \left[ A(0^{j+k}1...) - A(0^{j+\alpha}1...) \right] + A (10^{l+k}1...) - A(10^{j+\alpha}1...) \\
	&= \sum_{j=1}^l (a_{j+k} - a_{j+\alpha}) + d_{l+k} - d_{l+\alpha} \ .
	\end{align*}
	
	\begin{align*}
	W(  10^l | 1^k0 ) &= \left[ A(01^k0...) - A(0^{\alpha+1} 1 ...) \right] + \sum_{j=2}^l \left[ A(0^{j}1^k0) - A(0^{j+\alpha}1...) \right] + \\
	&+ \left[ A(10^l1^k0...) - A(10^{l+\alpha}1...) \right] \\
	&= b_k - a_{\alpha+1} + \sum_{j=2}^l (a_j - a_{j+\alpha}) + d_l - d_{l+\alpha} \ .
	\end{align*}
	
	\begin{align*}
	W(01^l|0^k1 )&= A(10^k1...)-A(10^\alpha1...) + \sum_{j=2}^l \left[ A(1^j0^k1...) - A(1^j0^\alpha1...) \right] \\
	&= A(10^k1...)-A(10^\alpha1...) \\
	&= d_{k} - d_\alpha \ .
	\end{align*}
	
	\begin{align}
	W( 01^l | 1^k0 ) &= A(1^{k+1}01...) - A(10^\alpha1...) + \sum_{j=2}^l \left[ A(1^{k+j}0...) - A(1^j0^\alpha1...) \right] + \\
	&+ A(01^{k+l}0...) - A(01^l0^\alpha1...) \\
	&=  c_{k+1} - d_\alpha + \sum_{j=2}^l(c_{k+j} - c_j) + (b_{k+l} - b_l). \label{oiay}
	\end{align}
	All other remaining cases are obtained from the previous one by limits. The results are summarized in Table \ref{dualpotentialtable2}.

	We denote the involution kernel based at $x \in [0^\alpha1]$ by $W_\alpha$. From all those equations, we can conclude two things:
	\begin{itemize}
		\item[(1)] there is no change in the convergence hypothesis needed to assert the existence of the involution kernel;
		\item[(2)] the sequence of functions $(W_\alpha)_{\alpha \in \N}$, assuming that they exists (equivalently, if one of them exists) converges uniformly to $W$, the involution kernel based at $x' = 0^\infty$.
	\end{itemize}

	\section{The dual Potential}
	\label{secdualpotential}
	
	Let $A \in R(\Omega)$ be a potential that admits involution kernel. 
	
	In \cite{BLT} (see also \cite{LMMS}) it was shown that 
	\begin{proposition} Given a continuous potential $A:\Omega \to \mathbb{R}$ and $x' \in \Omega$, if 
 $A^*$ is well defined when given by
 $$A^*(y) =A^*(y_1,y_2,...) = $$
 $$ A(y_1,x_1',x_2',...) + [A(y_2,y_1,x_1',x_2',...) - A(y_2,x_1',x_2',...)]+ ..$$
 \begin{equation}\label{elew} [A(y_n,...,y_2,y_1,x_1',x_2',...) - A(y_n,...,y_2,x_1',x_2',...)]+...,
 \end{equation}
 then $A^*$  is  a dual potential for $A$ (considering  the involution kernel  $W$ given by \eqref{poii}).
\end{proposition}
	
	 We are interested here in explicit expressions for the dual potential $A^*$ (which will be presented in Section \ref{kyu1}). Let $W$ be the involution kernel of $A$ based at the point $x' \in \Omega$. Then, for any ordered pair $(y|x) \in \hat{\Omega}$, 
	
$$A(\tau_y(x)) + W \circ \hat{\sigma}^{-1} (y|x) - W(y|x) = $$
	$$ A(\tau_y(x)) + \sum_{n=1}^\infty [ \hat{A} \circ \hat{\sigma}^{-n}(\sigma(y)|\tau_y(x)) -$$
	$$ \hat{A} \circ \hat{\sigma}^{-n}(\sigma(y)|x') ] - \sum_{n=1}^\infty \left[ \hat{A} \circ \hat{\sigma}^{-n}(y|x) - \hat{A} \circ \hat{\sigma}^{-n}(y|x') \right] = $$
	$$
  A(\tau_y(x)) +  \lim_{N \to \infty} \sum_{n=1}^N \left[ \hat{A} \circ \hat{\sigma}^{-(n+1)}(y|x) - \hat{A} \circ \hat{\sigma}^{-n}(\sigma(y)|x') \right. -$$
  $$ \left. \hat{A} \circ \hat{\sigma}^{-n}(y|x) + \hat{A} \circ \hat{\sigma}^{-n}(y|x') \right] \ ,$$
  
	since $\sum \alpha_n + \sum \beta_n = \sum(\alpha_n + \beta_n)$ if the series are convergent. But, for any natural number $N$,
	\begin{align*}
	A(\tau_y(x)) + \sum_{n=1}^N \left[ \hat{A} \circ \hat{\sigma}^{-(n+1)}(y|x) - \hat{A} \circ \hat{\sigma}^{-n}(\sigma(y)|x') - \hat{A} \circ \hat{\sigma}^{-n}(y|x) + \hat{A} \circ \hat{\sigma}^{-n}(y|x') \right] &= \\
	A(\tau_{y,N+1}(x)) + \sum_{n=1}^N \left[ \hat{A} \circ \hat{\sigma}^{-n}(y|x') - \hat{A} \circ \hat{\sigma}^{-n}(\sigma(y)|x') \right] \ .
	\end{align*}
	Therefore, if there is an involution kernel based on $x'$, the limit
	\begin{equation} \label{opiu}
	\lim_{N \to \infty} A(\tau_{y,N+1}(x)) + \sum_{n=1}^N \left[ \hat{A} \circ \hat{\sigma}^{-n}(y|x') - \hat{A} \circ \hat{\sigma}^{-n}(\sigma(y)|x') \right] \
	\end{equation}
	must exists and, by continuity of $A$, it does not depend on $x$. This shows that the function $\hat{A}^* :\hat{\Omega} \to \R$, given by
	\begin{align}
	A^* (y)\,=\,\hat{A}^*(y|x) := A(\tau_y(x)) + W \circ \hat{\sigma}^{-1} (y|x) - W(y|x) \ ,\label{opiu1}
	\end{align}
	is independent of $x$ and is a dual potential for $A$.
	
	\subsection{Calculation of the Dual Potential  for the Walters' family}
\label{kyu1}
	
		Let $W$ be an involution kernel for $A \in R(\Omega)$ and $A^*$ be the dual potential of $A$ associated to $W$.

	The tables \eqref{dualpotentialtable1}, \eqref{dualpotentialtable2} and \eqref{dualpotentialtable3},  to follow, shows our results in the calculation of dual potentials and helps our future analysis of the conditions for symmetry (see for instance Theorem \ref{puxa1}).

	\begin{definition} The potential $A$ is {\bf symmetrized} by $W$ if, and only if, $A^*  = A \circ \theta$. If $W$ is the involution kernel based at $x'$ and $A$ is symmetrized by $W$, we say that $A$ is symmetric relatively to $x'$.
		\label{chap1sec3subsec1defi1}
	\end{definition}
	
	\begin{table}[h]
	\begin{center}
		\begin{tabular}{c|c|c}
		\hline
			$x$ & $A$ & $A^* \circ \theta^{-1}$ \\
			\hline
			$0^\infty$ & $a$ & $a$ \\
			\hline
			$0^{n+1}1z$ & $a_{n+1} $ & $a$ \\
			\hline
			$0 1^n 0 z$ & $b_n$ & $a$ \\
			\hline
			$01^\infty$ & $b$ & $a$ \\
			\hline
			$1^\infty$ & $c$ & $c$ \\
			\hline
			$1^{n+1}0z$ & $c_{n+1}$ & $c_{n+1} + b_{n+1} - b_n$ \\
			\hline
			$10^n1z$ & $d_n$ & $(b_1 - a) + \sum_{j=2}^n (a_j - a) + d_n$ \\
			\hline
			$10^\infty$ & $d$ & $ d + (b_1 - a) + \sum_{j=2}^\infty (a_j - a)$ \\
			\hline
		\end{tabular}
		\caption{Dual Potential associated to the Involution Kernel based at $x' = 0^\infty$}
		\label{dualpotentialtable1}
		\end{center}
		\end{table}
	
	$\bullet$ Conditions for symmetry:
	The condition for symmetry is that the second column equals the third. The first two lines shows that $b = b_n = a_{n+1} = a$ for all natural numbers $n$. If these conditions hold, the other equations are trivially satisfied.

		\begin{table}[h]
		\begin{center}
		\begin{tabular}{c|c|c}
		\hline
			$x$ & $A$ & $A^* \circ \theta^{-1}$ \\
			\hline
			$0^\infty$ & $a$ & $a$ \\
			\hline
			$0^{n+1}1z$ & $a_{n+1}  $ & $a_{\alpha + n +1} + (d_{\alpha + n + 1} - d_{\alpha + n})$ \\
			\hline
			$0 1^n 0 z$ & $b_n$ & $a_{\alpha + 1} + (d_{\alpha + 1} - d_\alpha)$ \\
			\hline
			$01^\infty$ & $b$ & $a_{\alpha + 1} + (d_{\alpha + 1} - d_\alpha)$ \\
			\hline
			$1^\infty$ & $c$ & $c$ \\
			\hline
			$1^{n+1}0z$ & $c_{n+1}$ & $c_{n+1} + b_{n+1} - b_n$ \\
			\hline
			$10^n1z$ & $d_n$ & $d_\alpha + (b_1 - a_{\alpha+1}) + \sum_{j=2}^n (a_j - a_{\alpha + j}) + (d_n - d_{\alpha + n})$ \\
			\hline
			$10^\infty$ & $d$ & $ d_\alpha + (b_1 - a_{\alpha+1}) + \sum_{j=2}^\infty (a_j - a_{\alpha + j})$ \\
			\hline
		\end{tabular}
		\caption{Dual Potential associated to the Involution Kernel based at $x' = 0^\alpha 1 w$}
		\label{dualpotentialtable2}
		\end{center}
		\end{table}

	As should be expected, the first table is obtained from the second, making $\alpha \to \infty$. The first result of the above calculations is that the dual of a Walter's potential is also a Walter's potential.

	$\bullet$ Conditions for symmetry:
	The condition for symmetry is that the second column equals to the third. Then
	\begin{align*}
	b_n = a_{\alpha + 1} + (d_{\alpha + 1} - d_\alpha) = b \ ,
	\end{align*}
	that is, $b_n = b$ for all $n \in \N$, is a necessary condition.
	
	\begin{align*}
	a_{n+1} = a_{\alpha + n +1} + d_{\alpha + n +1} - d_{\alpha+ n} \iff a_{\alpha + n +1} - a_{n+1} = d_{\alpha + n} - d_{\alpha + n +1} \ , \ \forall \ n \in \N \ .
	\end{align*}
	
	Also, for all $n \in \N$
	\begin{align*}
	d_\alpha + (b_1 - a_{\alpha+1}) + \sum_{j=2}^n (a_j - a_{\alpha + j}) = d_{\alpha + n}  \ ,
	\end{align*}
	and
	\begin{align}
	d = d_\alpha + b - a_{\alpha + 1} + \sum_{j=2}^\infty (a_j - a_{\alpha+j}) \ .
	\label{chap1sec3subs1eq6}
	\end{align}
	Then the conditions for symmetry in this case is the validity of the system of equation
	\begin{align}
	\begin{cases}
	a_{\alpha + n +1} - a_{n+1} = d_{\alpha + n} - d_{\alpha + n +1} \\
	b_n = a_{\alpha + 1} + (d_{\alpha + 1} - d_\alpha) = b \\
	d_{\alpha + n} = d_\alpha + (b_1 - a_{\alpha+1}) + \sum_{j=2}^n (a_j - a_{\alpha + j}) \\
	d = d_\alpha + b - a_{\alpha + 1} + \sum_{j=2}^\infty (a_j - a_{\alpha+j})
	\end{cases} \ ,
	\label{chap1sec3subs1eq7}
	\end{align}
	$\forall \ n \in \N$.
     
        \begin{table}
	    \begin{center}
		\begin{tabular}{c|c|c}
		\hline
			$x$ & $A$ & $A^* \circ \theta^{-1}$ \\
			\hline
			$0^\infty$ & $a$ & $a$ \\
			\hline
			$0^{n+1}1z$ & $a_{n+1}  $ & $a_{n +1} + (d_{ n + 1} - d_{n})$ \\
			\hline
			$0 1^n 0 z$ & $b_n$ & $ b_ \alpha + (d_1 - c_{\alpha + 1}) + \sum_{j=2}^n (c_j - c_{j+\alpha}) + (b_n - b_{n+\alpha}) $ \\
			\hline
			$01^\infty$ & $b$ & $b_\alpha + (d_1 - c) + \sum_{j=2}^\infty ( c_j - c)$ \\
			\hline
			$1^\infty$ & $c$ & $c$ \\
			\hline
			$1^{n+1}0z$ & $c_{n+1}$ & $c_{\alpha + n + 1} + (b_{\alpha + n + 1} - b_{\alpha + n})$ \\
			\hline
			$10^n1z$ & $d_n$ & $c_{\alpha + 1} + (b_{\alpha + 1} - b_\alpha)$ \\
			\hline
			$10^\infty$ & $d$ & $ c_{\alpha + 1} + (b_{\alpha + 1} - b_\alpha)$ \\
			\hline
		\end{tabular}
		\caption{Dual Potential associated to the Involution Kernel based at $x' = 1^\alpha 0 w$}
		\label{dualpotentialtable3}
	\end{center}
	\end{table}
	
	\begin{align}
	\begin{cases}
	c_{\alpha + n +1} - c_{n+1} = b_{\alpha + n} - b_{\alpha + n +1} \\
	d_n = c_{\alpha + 1} + (b_{\alpha + 1} - b_\alpha) = d \\
	b_{\alpha + n} = b_\alpha + (d_1 - c_{\alpha+1}) + \sum_{j=2}^n (c_j - c_{\alpha + j}) \\
	b = b_\alpha + d - c_{\alpha + 1} + \sum_{j=2}^\infty (c_j - c_{\alpha+j})
	\end{cases} \ ,
	\label{chap1sec3subs1eq8}
	\end{align}
	$\forall \ n \in \N$.

	\begin{table}
	\begin{center}
		\begin{tabular}{c|c|c}
		\hline
			$x$ & $A$ & $A^* \circ \theta^{-1}$ \\
			\hline
			$0^\infty$ & $a$ & $a$ \\
			\hline
			$0^{n+1}1z$ & $a_{n+1}  $ & $a_{n +1} + (d_{ n + 1} - d_{n})$ \\
			\hline
			$0 1^n 0 z$ & $b_n$ & $ b + (d_1 - c) + \sum_{j=2}^n (c_j - c) + (b_n - b) $ \\
			\hline
			$01^\infty$ & $b$ & $b + (d_1 - c) + \sum_{j=2}^\infty ( c_j - c)$ \\
			\hline
			$1^\infty$ & $c$ & $c$ \\
			\hline
			$1^{n+1}0z$ & $c_{n+1}$ & $c $ \\
			\hline
			$10^n1z$ & $d_k$ & $c$ \\
			\hline
			$10^\infty$ & $d$ & $c$ \\
			\hline
		\end{tabular}
		\caption{Dual Potential associated to the Involution Kernel based at $x' = 1^\infty$}
		\label{dualpotentialtable4}
	\end{center}
	\end{table}
	In analogy with the first case, the conditions for symmetry here are $c_{n+1} = d_n = c \ \forall \ n \in \N$.
	
	The following Theorem summarizes our conclusions:
	
	\begin{theorem}  \label{puxa1}
		Suppose that $A \in R(\Omega)$ admits involution kernel. If $b_n = a_{n+1} = a$ for all $n \in \N$, then $A$ is symmetric relatively to the involution kernel based at $x' = 0^\infty$. Let $\alpha \in \N$. If
		\begin{align*}
		\begin{cases}
		a_{\alpha + n +1} - a_{n+1} = d_{\alpha + n} - d_{\alpha + n +1} \\
		b_n = a_{\alpha + 1} + (d_{\alpha + 1} - d_\alpha) = b \\
		d_{\alpha + n} = d_\alpha + (b_1 - a_{\alpha+1}) + \sum_{j=2}^n (a_j - a_{\alpha + j}) \\
		d = d_\alpha + b - a_{\alpha + 1} + \sum_{j=2}^\infty (a_j - a_{\alpha+j})
		\end{cases} \ 
		\end{align*}
		holds $\forall \ n \in \N$, then $A$ is symmetric relatively to $x' = 0^\alpha 1 z$. If $d_n = c_{n+1} = c$ for all natural numbers $n$ then $A$ is symmetric relatively to the point $x' = 1^\infty$. If 
		\begin{align*}
		\begin{cases}
		c_{\alpha + n +1} - c_{n+1} = b_{\alpha + n} - b_{\alpha + n +1} \\
		d_n = c_{\alpha + 1} + (b_{\alpha + 1} - b_\alpha) = d \\
		b_{\alpha + n} = b_\alpha + (d_1 - c_{\alpha+1}) + \sum_{j=2}^n (c_j - c_{\alpha + j}) \\
		b = b_\alpha + d - c_{\alpha + 1} + \sum_{j=2}^\infty (c_j - c_{\alpha+j})
		\end{cases} 
		\end{align*}
		holds for all $n \in \N$ then $A$ is symmetric relatively to the point $x' = 1^\alpha 0z$.
		\label{teo3}
	\end{theorem}
	
	The above Theorem allow us to show that there exists potentials in $R(\Omega)$ for which there is no involution kernel, {\bf accordingly to Definition} \ref{chap1defi2}, that makes it symmetric. In fact, take four distinct real number $a,b,c$ and $d$, and define $A$ with the constant sequences $a_n = a$, $b_n = b$, $c_n = c$ and $d_n = d$. By Theorem \ref{teo2}, $A$ has involution kernel based at any point $x' \in \Omega$. For any $n,\alpha \in \N$,
	\begin{align*}
	\begin{cases}
	b_n = b \neq a = a_n \implies \text{$A$ is not symmetric relatively to $0^\infty$}\\
	b_n \neq a_{\alpha+1} = a_{\alpha+1} + (d_{\alpha+1} - d_\alpha) \implies \text{$A$ is not symmetric relatively to $0^{\alpha}1z$}\\
	d_n = d \neq c = c_{n+1} \implies \text{$A$ is not symmetric relatively to $1^\infty$}\\
	d_n \neq c_{\alpha + 1} = c_{\alpha + 1} + (b_{\alpha + 1} - b_\alpha) \implies \text{$A$ is not symmetric relatively to $1^\alpha 0z$}
	\end{cases}
	\end{align*}
	so $A$ cannot be symmetric relatively to some point.
	
	\begin{remark}
		The results in this paper are restricted to a specific class of involution kernels; those which comes from a limit as in \eqref{chap1eq2}. Thus, our results have some restrictions, which we illustrate with an example. It comes from a simple observation: If $A_1$ and $A_2$ are potentials, $W_1$ is an involution kernel for $A_1$ and $W_2$ is an involution kernel for $A_2$, then $W_1 + W_2$ is an involution kernel for $A_1 + A_2$. Also, if $k \in \R \setminus \{ 0 \}$, $kW_1$ is an involution kernel for $k A_1$.
		
		Suppose that $A: \Omega \to \R$ is a potential in $R(\Omega) $ depending only on the two first coordinates. Given a point $x' = (x'_n)_{n \in \N} \in \Omega$,
		$$
		\sum_{n \in \N} [ \hat{A} \circ \hat{\sigma}^{-n}(y|x) - \hat{A} \circ \hat{\sigma}^{-n}(y|x') ] =$$
		$$ A(y_1 x_1...) - A(y_1x'_1...) + A(y_2y_1x_1...) - A(y_2y_1x'_1...) =$$
	$$	A(y_1x_1...) - A(y_1x'_1), $$
		and this implies that
		\begin{align*}
		\hat{A}^*(y|x) = A(y_2y_1x) + \left[ A(y_1x') - A(y_2x') \right].
		\end{align*}
		Suppose now that $A$ is a indicator function $\chi_{[a_0,a_1)}$ of a cylinder $[a_0a_1)$. The above expression shows that, taking $x' \notin [a_1)$, 
		\begin{align*}
		\hat{\chi}^*_{[a_0a_1)} (y|x) = \chi_{[a_0a_1)} (y_2,y_1 x) = \chi_{(a_1a_0]}(y) \ ,
		\end{align*}
		i.e., it is possible to choose $x'$ in such a way that $\chi^*_{[a_0a_1)} = \chi_{(a_1a_0]}$. 
		
		Given a $2\times 2$ line stochastic matrix $\P = [p_{ij}]$, the associated stationary Markovian measure over $\Omega$ can be obtained via Thermodynamic Formalism by considering the potential
		$A = p_{11} \chi_{[00)} + p_{12} \chi_{[01)} + p_{21} \chi_{[10)} + p_{22} \chi_{[11)}$. The potential $A$ thus defined is in $R(\Omega)$. For each $\chi_{[a_0a_1)}$ we can choose $x'(a_0a_1) \in \Omega$ such that $\chi^*_{[a_0a_1)}$, the dual with respect to the involution kernel based at $x'(a_0,a_1)$, is $\chi_{(a_1a_0]}$. Considering the observation made in the beginning of the remark, we conclude that
		\begin{align*}
		B := p_{11} \chi_{(00]} + p_{12} \chi_{(10]} + p_{21} \chi_{(01]} + p_{22} \chi_{(11]} 
		\end{align*}
		is a dual potential for $A$ (see Section 5 in \cite{BLLCo}). If $\P$ is row-stochastic, we have an associate markov process in $\Omega^-$ defined by the column-stochastic matrix $\P^T$.
		
		Besides $B$ is a dual potential for $A$, it may not be obtained by an involution kernel of the form \eqref{chap1eq2}: for each indicator function $\chi_{[a_0a_1)}$ we choose a $x'(a_0a_1)$ to obtain the kernel, and $B$ is obtained through the linear combination of these kernels. The sum of two kernels of type \eqref{chap1eq2} may not be of the same type unless the respective base points are the same. Thus, our results do not apply to $B$ as it does not apply, in general, to dual potential obtained by linear combinations of involution kernels of type \eqref{chap1eq2}.
		
		The important result which is announced in Corollary \ref{normpotentialscor2} has also the same restriction
	    \end{remark}

	\section{Twist Condition} \label{twi}
	
	In the works  \cite{LOT} and \cite{LOS}, the involution kernel appears related to an   Ergodic Transport Problem. The involution kernel can be seen as a natural cost function on $\hat{\Omega}$. In the symbolic setting, the Twist Condition plays the role of convexity. 
	\begin{definition}
	Considering the lexicographic order $<$ in $\Omega$ and $\Omega^-$, the involution kernel $W$ is said to satisfy the Twist Condition if and only if $y < y'$ and $x < x'$ implies
	\begin{align*}
	W(y|x) + W(y'|x') < W(y|x') + W(y'|x') \ .
	\end{align*}
	\end{definition}
	
	For the involution kernels in this paper, the twist condition cannot be satisfied, as we show now.
	Let $a \in [1^s0], \ a' \in [0^{s+p}1],\ b \in [0^{k+q}1]$ and $b' \in [0^k1]$. Then, if the involution kernel based at $x' = 0^\alpha 1 z$ satisfies the twist condition, 
	$$
	W(...01^s | 0^{k+q}1...) + W(...01^{s+p}|0^k1...) < $$
	$$W (...01^s|0^k1...) + W(...01^{s+p}|0^{k+q}1...) $$
	
	This is equivalent to
	\begin{align}
	( d_{k+q} - d_\alpha) + (d_k - d_\alpha) < (d_k - d_\alpha) + (d_{k+q} - d_\alpha) \ .
	\label{twisteq1}
	\end{align}
	
	If we change the base point of the involution kernel to be $x' = 0^\infty$, the above inequality changes only by the substitution $d_\alpha \to d$. To consider basis points in the cylinder $[1]$ would lead to a similar contradiction: considering $b \in [1^s0], \ b' \in [1^{s+p}0],\ a \in [0^{k+q}1]$ and $a' \in [0^k1]$ and the involution kernel based at $x' = 1^\alpha 0 z$, it holds
	$$ W(...10^{k+q}|1^s0...) + W (...10^k | 1^{s+p}0...) =$$
	\begin{align*}
	 b_{k+q} - b_\alpha = W(...10^{k+q}|0^{s+p}1...) + W(...10^k |1^s0...).
	\end{align*}
	Therefore, the  twist condition cannot be satisfied.
	
	\begin{definition} \label{kil}
		A cost function $c: \hat{\Omega} \to \R$ is said to satisfy a relaxed twist condition if, for any pair $(a,b),(a'b') \in \hat{\Omega}$ with $a < a'$ and $b < b'$, 
		\begin{align*}
		c(a,b) + c(a',b') \leq c(a,b') + c(a',b) \ .
		\end{align*}
	\end{definition}
    
    We can define a class of Walters Potentials for which the involution kernel satisfies the relaxed twist condition. Indeed, let $W$ be the involution kernel based at the point $0^\alpha1z$ of a Walters potential $A$ defined by the sequences $(a_n)_{n \in \N}$,$(b_n)_{n \in \N}$,$(c_n)_{n \in \N}$ and $(d_n)_{n \in \N}$. 
    
   Define, for each $n \in \N$, the sequences $\left( \Delta^n_i \right)_{i \in \N}$ and $\left( \Gamma^n_i \right)_{i \in \N}$ via $\Delta^n_{i} := d_{n+i} - d_i$, $\Gamma^m_i := b_{n+i} - b_i$.

    \begin{theorem} \label{kjg}
    	If $(a_n)_{n \in \N},\ (c_n)_{n \in \N}$ and, for each $m \in \N$, $\left( \Delta^m_n\right)_{n \in \N}$ and $\left( \Gamma^m_n \right)_{n \in \N}$ are decreasing, and if $(d_n)_{n \in \N}$ and $(b_n)_{n \in \N}$ are sub-additive, the involution kernel based at $x$ satisfies the relaxed twist condition, for any $x \in \Omega$.
    \end{theorem}
    
	\section{Normalized Potentials on $G(\Omega)$} \label{nono}
	
	Suppose $A$ is normalized. If $A^*$ is the dual potential associated to the point $0^\alpha 1 z$, then for it to be normalized, accordingly to Table \ref{dualpotentialtable2}, the following equations must be satisfied:
	$$ \exp \left[ d_\alpha + b_1 - a_{\alpha +1} + \sum_{j=2}^{n} (a_j - a_{\alpha + j}) + d_n - d_{n+\alpha} \right] =$$
	\begin{align}
	 1 - \exp \left( a_{\alpha + n + 1} + d_{\alpha + n + 1} - d_{\alpha + n} \right) \ ,
	\label{normpotentialeq1}
	\end{align}
	\begin{align}
		\exp \left( a_{\alpha+1} + d_{\alpha+1} - d_\alpha \right) = 1 - \exp \left( c_{n+1} + b_{n+1} - b_n \right) \ .
		\label{normpotentialeq2}
	\end{align}
	
	The normalization condition  of $A$ reduces its degree of freedom; the sequences $(b_n)_{n \in \N}$ and $(d_n)_{n \in \N}$ becomes entirely defined by the sequences $(a_{n+1})_{n \in \N}$ and $(c_{n+1})_{n \in \N}$. We show that normalization of $A^*$ defines $(c_{n+1})_{n \in \N}$ in terms of $(a_{n+1})_{n \in \N}$.
	
	\begin{proposition}
		If $A$ and $A^*$ are normalized, then
		\begin{align*}
		e^{c_{n+2}} = 1 - \left(1- e^{a_{\alpha+1} + d_{\alpha + 1} - d_\alpha} \right) \. \left( e^{-c_{n+1}} - 1 \right) \ .
		\end{align*}
		\label{normpotentialsprop1}
	\end{proposition}

	\begin{proposition}
		If $A$ and $A^*$ are normalized, then
		\begin{align*}
		1-e^{c_2} = \frac{1 - e^{a_{\alpha + n +1} + a_{\alpha + n +2}}}{\left(e^{-a_{\alpha +1}} - 1 \right) \left( 1 - e^{a_{n+1}} \right) \exp \left[ \sum_{j=2}^n (a_j - a_{\alpha+j}) \right]} \ ,\ \forall n \in \N \ .
		\end{align*}
	\label{normpotentialsprop2}
	\end{proposition}

\begin{corollary}
	If $A$ and $A^*$ are normalized, then the sequence
	\begin{align*}
	\frac{1 - e^{a_{\alpha + n+1}+a_{\alpha+n+2}}}{\left(e^{-a_{\alpha +1}} - 1 \right) \left( 1 - e^{a_{n+1}} \right) \exp \left[ \sum_{j=2}^n (a_j - a_{\alpha+j}) \right]}
	\end{align*}
	is constant.
	\label{normpotentialscor1}
\end{corollary}

	Then we can reach the following conclusion: there are normalized potentials $A \in R(\Omega)$ for which there is no dual normalized potential. In fact, $A$ is normalized if and only if the system
	\begin{align}
	\begin{cases}
	d_n = \log (1 - e^{a_{n+1}}) \\
	b_n = \log (1 - e^{c_{n+1}})
	\end{cases} \ \forall \ n \in \N \ ,
	\label{normpotentialeq6}
	\end{align}
	holds and there exists $r \in (0,1)$ such that $e^{a_{n+1}} \in (c,1)$ and $e^{c_{n+1}} \in (0,1-c)$ for all $n \in \N$%(\cite{walters1}).
	
	For any choice of sequences $(a_{n+1})_{n \in \N}$ and $(c_{n+1})_{n \in \N}$ satisfying theses respective constraints, with 
	\begin{align*}
	\sum_{n \in \N} (a - a_n) \ \text{and} \ \sum_{n \in \N} (c - c_n)
	\end{align*}
	convergent, defining $b_n$ and $d_n$ by the above equations, we get a normalized potential, which may not satisfy the conclusion of Proposition \ref{normpotentialsprop1}. 
	
	\begin{example}
		Let $a_n = d_n = \log 1/2$ for all $n \in \N$. Let $c_n =\log 2/3$ and $b_n = \log 1/3$, for all $n \in \N$.Since $1/2 > 1/3$, $1/3 < 2/3$, and the equations \eqref{normpotentialeq6} are satisfied, the potential derived from these sequences is normalized. We show that Proposition \ref{normpotentialsprop1} does not hold. 
		\begin{align*}
		1 - \left(1- e^{a_{\alpha+1} + d_{\alpha + 1} - d_\alpha} \right) \. \left( e^{-c_{n+1}} - 1 \right) &= 1 - \left( 1 - \frac{1}{2} \right) \left( \frac{3}{2} - 1 \right) \\
		&= 1 - \frac{1}{4} \\
		&= \frac{3}{4} \\
		&\neq \frac{2}{3} \\
		&= e^c \ .
		\end{align*}
	\end{example}

	On the other hand, if $A$ satisfies the Ruelle's Operator Theorem, its dual also satisfies it.  
	
	Let $(\lambda, h)$ be a positive eigenpair for the Ruelle Operator associated with $A$, and let ($\lambda^*,h^*$) be a positive eigenpair for the Ruelle Operator associated with $A^*$. Moreover, take $\lambda$ and $\lambda*$ as the greatest eigenvalue of $\L_A$ and $\L_{A^*}$, respectively. We know that $\lambda = \lambda^*$. 
	
	Given a Holder potential $A$ and its equilibrium probability $\mu_A$,
	there exists a unique positive function $J:\Omega \to (0,1)$,  such that,
	$$ \mathcal{L}^*_{\log J} (\mu_A)= \mu_A.$$
	
	For a proof see \cite{PP}.  We call $J$ the Jacobian of the probability $\mu_A$.

	We study the relation between $h$ and $h^*$ and the Jacobians of the equilibrium states of the potentials $A$ and $A^*$. Fixing $A \in R(\Omega)$ with involution kernel $W$ and denoting $A^*$ the corresponding dual relatively to the involution kernel based at the point $0^\delta 1w$, we prove 
	\begin{theorem}
		Let $\mu$ and $\mu^*$ be the equilibrium states of $A$ and $A^*$, respectively. Then the Jacobians of $J$ of $\mu$ and $J^*$ of $\mu^*$ satisfies $J = J^* \circ \theta^{-1}$.
		\label{normpotentialtheo1}
	\end{theorem}

\begin{corollary}
	Let $\theta^{-1}_*$ be the pull-back of $\mu^*$ by $\theta^{-1}$. In the above conditions, $\mu = \theta^{-1}_*\mu^*$.
	\label{normpotentialscor2}
\end{corollary}
\begin{proof}
	Since we are supposing that the series $\sum (a_n - a)$ and $\sum (c_n - c)$ converges, $A$ and $A^*$ has unique equilibrium states; $\mu$ and $\mu^*$, respectively. Let $\overline{A}$ and $\overline{A}^*$ be the normalized potentials associated to $A$ and $A^*$, respectively. As described in \cite{Wal2}, for any continuous functions $f: \Omega \to \R$, 
	\begin{align*}
	\L^n_{\overline{A}} f \to \int f \ d\mu \ , \ \text{and} \ \ \L^n_{\overline{A}^*} f \to \int f \circ \theta^{-1} \ d\mu^*  = \int f \ d \left( \theta^{-1}_* \mu^* \right) \ ,
	\end{align*}
	in the uniform topology (the right side of the arrows denotes, with language abuse, constant functions which assumes that values). Since $A$ and $A^*$ has the same Jacobians, $\L_{\overline{A}} = \L_{\overline{A}^* \circ \theta^{-1}}$, and thus $\int f \ d\mu = \int f \ d \left( \theta^{-1}_* \mu^* \right)$ for any continuous function $f: \Omega \to \R$. From Riesz - Markov Representation Theorem, it follows that $\mu = \theta^{-1}_* \mu^*$.
\end{proof}

\begin{remark}
	Corollary \ref{normpotentialscor2} can be generalized: in the class $R(\Omega)$ if the equilibrium states of two potentials have the same Jacobians, then the equilibrium states are equal. 
\end{remark}

\section{A generalization} \label{lala}

Let $\M(\sigma)$ and $\M(\hat{\sigma})$ be the set of Borel invariant probability measures for $\sigma$ and $\hat{\sigma}$, respectively. We show the existence of a a bijection between these sets, which preserves entropy. Given $\mu \in \M(\sigma)$, define $\hat{\mu}$, the {\bf extension of $\mu$ to $\hat{\Omega}$}, as being the measure in $\M(\sigma)$ satisfying
\begin{align*}
\hat{\mu}(<a_n...a_1|b_1...b_m>) = \mu([a_n....a_1b_1...b_m]) \ ,
\end{align*}
for any $(a_n,...,a_1,b_1,...,b_m) \in \bigcup_{k \in \N} \A^{(k)}$. Now, there is a natural inclusion $i$ of Borel $\sigma$-algebra $\B(\Omega)$ into the Borel $\sigma$-algebra $\B(\hat{\Omega})$. It satisfies $i([a_1...a_n]) = |a_1...a_n>$ for all words in the alphabet $\A$. Given $\hat{\mu} \in \M(\hat{\sigma})$, let $\mu$ be the pull-back by the map $i$ of the restriction of $\hat{\mu}$ to $i(\B(\Omega))$. The map $\hat{\mu} \mapsto \mu$ thus defined is the inverse of the map $\mu \mapsto \hat{\mu}$ defined previously.

The natural  extension of a $\sigma$-invariant probability on $\Omega$ to a $\hat{\sigma}$-invariant probability on $\hat{\Omega}$ is unique.
Results related to this section appear in Section 7 in  \cite{LM}.

\begin{proposition}
	For all $\hat{\mu} \in \M(\hat{\sigma})$, $h_{\hat{\mu}}(\hat{\sigma}) = h_\mu(\sigma)$.
	\label{sec7prop1}
\end{proposition}
\begin{proof}
	Let $\hat{\P} = \left\lbrace |a>;\ a \in \A \right\rbrace$ and $\P = \left\lbrace [a];\ a \in \A \right\rbrace$. $\hat{\P}$ is a bilateral generating partition for $(\hat{\Omega}, \hat{\sigma})$ (i.e., $\bigcup_{n \in \N} \hat{\P}^{\pm n} = \bigcup_{n \in \N} \bigvee_{j=-n}^n \hat{\sigma}^j (\hat{\P})$ generates the Borel $\sigma$-algebra of $\hat{\Omega}$), while $\P$ is a unilateral generating partition for $(\Omega,\sigma)$ ($\bigcup_{n \in \N} \P^n = \bigcup_{n \in \N} \bigvee_{j=0}^{n-1}\sigma^{-j}(\P)$ generates de Borel $\sigma$-algebra of $\Omega$). By Kolmogorov-Sinai's Theorem, 
	\begin{align*}
	h_{\hat{\mu}}(\hat{\sigma}) = h_{\hat{\mu}}(\hat{\sigma},\hat{\P}) \ 
	\text{and} \ h_\mu(\sigma) = h_\mu(\sigma,\P) \ .
	\end{align*} 
	Invariance of the measures under the respective dynamics, and the relation between $\mu$ and $\hat{\mu}$, implies $h_{\hat{\mu}}(\hat{\sigma},\hat{\P}) = h_\mu(\sigma,\P)$.
\end{proof}

If $P: {\mathcal C}(\Omega) \to \R$ and $\hat{P}: {\mathcal C}(\hat{\Omega}) \to \R$ are the topological pressures, Proposition \ref{sec7prop1} implies that $\hat{P}(\hat{A}) = P(A)$ for all $A \in {\mathcal C}^(\Omega)$. Thus, the restriction of any equilibrium state of $\hat{A}$ to $\Omega$ is an equilibrium state for $A$. The map $\hat{\mu} \mapsto \mu$ is a bijection. So, if $A$ has unique equilibrium state, $\hat{A}$ also has unique equilibrium state. 

Analogous results holds for the pair $\hat{A}^*$ and $A^*$. But, since
\begin{align*}
\int \hat{A}^*  \ d\hat{\mu} = \int \left[ \hat{A} \circ \hat{\sigma}^{-1} + W \circ \hat{\sigma}^{-1} - W \right] \ d \hat{\mu} = \int \hat{A} \ d\hat{\mu} \ \forall \ \hat{\mu} \in \M(\hat{\sigma}) \ ,
\end{align*}
these two potentials have the same equilibrium states. Therefore, uniqueness of equilibrium state for $A$ implies uniqueness of equilibrium state for $\hat{A}$, which implies uniqueness of equilibrium state for $\hat{A}^*$, which implies uniqueness of equilibrium state for $A^*$. Also, 
\begin{align*}
\hat{\mu}_{A^*} = \hat{\mu}_A  \ .
\end{align*}

\begin{corollary} \label{veve}
	If $\mu_A ([a_1...a_n]) = \mu_A([a_n...a_1])$ for all $(a_1,...,a_n) \in \bigcup_{m \in \N} \A^m$, then $\theta^{-1}_*\mu_{A^*} = \mu_A$.
\end{corollary}

Corollary 2.3 in \cite{Wal2} claims that equilibrium probabilities for normalized potentials on the Walters' family satisfy the hypothesis of the above corollary. It follows from \cite{Jiang} that the entropy production in this case is $0$.

\smallskip

The present work is part of the Master Dissertation of L. Y. Hataishi in Prog. Pos. Grad. em 
Matematica, UFRGS.
\smallskip

%Data Availability: The authors declare that data sharing is not applicable to this article as no datasets were
%generated or analyzed during the current study


\begin{thebibliography}{99}


 \bibitem{BLT} A. Baraviera, A. O. Lopes and Ph. Thieullen,
A Large Deviation Principle for Gibbs states of H\"older potentials: the zero temperature case. \emph{ Stoch. and  Dyn.}\, (6), 77-96, (2006).

\bibitem{BLM} A. Baraviera, A. O. Lopes A and J. Mengue, On the selection of subaction and  measure for a subclass of  Walters's potentials, Volume 33, issue 05, pp. 1338--1362,  \emph{Erg. Theo. and Dyn. Syst.} (2013)

\bibitem{BLL1}
A. Baraviera, R. Leplaideur and  Artur O. Lopes,  The potential point of view for renormalization. Stoch. Dyn. 12, no. 4, 1250005, 34 pp (2012)

\bibitem{BLL} A. Baraviera, R. Leplaideur and A. O. Lopes, Selection of measures for a potential with
two maxima at the zero temperature limit, \emph{SIAM Journal on Applied Dynamical Systems}, Vol. 11, n 1, 243-260 (2012)


\bibitem{BLLCo}	A. Baraviera and R. Leplaideur and A. O. Lopes, Ergodic Optimization, Zero Temperature Limits and the Max-Plus Algebra, mini-course in XXIX Coloquio Brasileiro de Matematica (2013)

\bibitem{Bha}
P. Bhattacharya and M. Majumdar, Random Dynamical Systems. Cambridge Univ. Press, 2007.

\bibitem{CL}	L. Cioletti and A. O. Lopes, Correlation Inequalities and Monotonicity Properties of the Ruelle Operator,
Stoch. and Dyn, 19, no. 6, 1950048, 31 pp (2019)


\bibitem{CL1}  L. Cioletti and A. O. Lopes,
Phase Transitions in One-dimensional Translation Invariant Systems: a Ruelle Operator Approach,
Journ. of Statistical Physics, 159 - Issue 6, 1424-1455 (2015)


\bibitem{CLS}  L. Cioletti, A. O. Lopes and  M. Stadlbauert, Ruelle Operator for Continuous Potentials and DLR-Gibbs Measures,
Disc and .Cont. Dyn. Syst. A Vol 40, N. 8, 4625-4652 (2020)

\bibitem{CSS} L. Cioletti, L. Melo, L. Ruviaro and E. A. Silva, On the dimension of the space of harmonic functions on transitive shift spaces,
Advances in Mathematics, 2021, 385, 107758

\bibitem{CDLS} L. Cioletti, M. Denker, A. O. Lopes and M. Stadlbauer,
Spectral Properties of the Ruelle Operator for Product Type Potentials on Shift Spaces,
Journal of the London Mathematical Society - Volume 95, Issue 2, 684-704 (2017)

\bibitem{CHLS} L. Cioletti, L. Hataishi and M. Stadlbauert,
Spectral Triples on Thermodynamic Formalism and Dixmier Trace Representations of Gibbs Measures: theory and examples, arXiv

\bibitem{CLO} G. Contreras, A. O. Lopes  and E. Oliveira, Ergodic Transport Theory, periodic maximizing probabilities and the twist condition, "Modeling, Optimization, Dynamics and Bioeconomy I", Springer Proceedings in Mathematics and Statistics, Volume 73, Edit. David Zilberman and Alberto Pinto, 183-219 (2014)

\bibitem{FLO} H. H. Ferreira, A. O. Lopes and E. R. Oliveira,
Explicit examples in Ergodic Optimization, Sao Paulo Journal of Math. Sciences - Vol 14 - pp 443-489 (2020)

\bibitem{FL} A. Fisher and A. O. Lopes,
Exact bounds for the polynomial decay of
correlation, 1/f noise and the central limit theorem for a
non-H\"older Potential,
Nonlinearity, Vol 14, Number 5, pp 1071--1104 (2001).

\bibitem{Galla} 
G. Gallavotti, Fluctuation patterns and conditional reversibility in nonequilibrium systems, Annales de
l’I.H.P. Physique theorique Vol. 70, Issue: 4,  429443 (1999) 

\bibitem{GLP}  
P. Giulietti, A. O. Lopes and V. Pit,
Duality between Eigenfunctions and Eigendistributions of Ruelle and Koopman operators via an integral kernel,
Stoch. and Dynamics, 16 , 1660011 - 22 -pages - vol 3 (2016).


\bibitem{dissH} L. Y. Hataishi, 	Spectral Triples em Formalismo
Termodin\^amico
e
Kernel de Involu\c c\~ao  para potenciais Walters, Master Dissertation, Pos. Grad.  Mat - UFRGS (2020)

\bibitem{Hof}
F. Hofbauer,  Examples for the nonuniquenes of the equilibrium state,  Transactions AMS, 228, 133–141, (1977)

\bibitem{Jiang}
D.-Q. Jiang, M. Qian and M.-P. Qian, Entropy production and information gain in axiom-A
systems, Commun. Math. Phys., 214, 389–-409  (2000)


\bibitem{Lo1} A. O. Lopes,
The Zeta Function, Non-Differentiability
of Pressure and The Critical Exponent of Transition, Advances in Mathematics, Vol. 101,  133-167
(1993)

\bibitem{LMMS} A. O. Lopes, J. K. Mengue, J. Mohr and  R. R. Souza,
Entropy and Variational Principle for  one-dimensional Lattice Systems with a general a-priori probability: positive and zero temperature,  Erg. Theory and Dyn Systems, 35 (6), 1925--1961 (2015)

\bibitem{LopR} A. O. Lopes,
A general renormalization procedure on the one-dimensional lattice and decay of correlations, Stoch. and Dyn., Vol. 19, No. 01, 1950003 (2019)

\bibitem{LM} A. O. Lopes and J. K. Mengue,
On information gain, Kullback-Liebler
divergence, entropy production and the
involution kernel, Disc. Cont. Dyn. Syst. Series A. Vol. 42, No. 7,   3593–-3627
(2022)


 
 
 \bibitem{LOS} A. O. Lopes,  E. R. Oliveira and D. Smania,
 Ergodic Transport Theory and Piecewise Analytic Subactions for Analytic Dynamics,
Bull. of the Braz. Math Soc. Vol 43 (3) 467-512 (2012)




\bibitem{LOT}
A. O. Lopes, E. R. Oliveira and Ph. Thieullen,
The dual potential, the involution kernel and transport in ergodic optimization,
Dynamics, Games and Science - International Conference and
Advanced School Planet Earth DGS II, Portugal (2013), Edit. J-P Bourguignon, R. Jelstch,
A. Pinto and M. Viana, Springer Verlag, 357-398 (2015)

\bibitem{LOS} A. O. Lopes, E. R. Oliveira and D. Smania,
Ergodic Transport Theory and Piecewise Analytic Subactions for Analytic Dynamics,
Bull. of the Braz. Math Soc. Vol 43 (3) 467-512 (2012)

\bibitem{LV}A. O. Lopes and V. Vargas, Gibbs States and Gibbsian Specifications on the space $\mathbb{R}^\mathbb{N}$.
Dyn. Systems., Volume 35, Issue 2, 216-241 (2020)


\bibitem{L3} A.O. Lopes,
\emph{Thermodynamic Formalism, Maximizing Probabilities and Large Deviations}, Notes on line -  UFRGS.

\bibitem{Lofi}
A. O. Lopes, A first order level-2 phase transition in thermodynamic formalism, J. Stat. Phys.,
60 Nos 3/4 (1990), 395--411

%\bibitem{Manebook} R.  Ma\~n\'e, Ergodic Theory and Differentiable Dynamics, Springer Verlag (1987)

\bibitem{Jiang} 
Da-quan Jiang, Min Qian and Min-ping Qian, Entropy Production and
Information Gain in Axiom-A Systems, Commun. Math. Phys., 214, 
389-409 (2000).

\bibitem{Maes} 
C. Maes and K. Netocny, Time-Reversal and Entropy, Journal of Statistical Physics, Vol. 110, Nos.
1/2, (2003)

\bibitem{Mengue} J. K. Mengue,
Large deviations for equilibrium measures and selection of subaction,
Bulletin of the Brazilian Mathematical Society, Vol. 49, no. 1, p. 17-42 (2018).

\bibitem{Mit}
T. Mitra, Introduction to Dynamic Optimization Theory, Optimization and Chaos, Editors: M. Majumdar, T. Mitra
and K. Nishimura, Studies in Economic Theory, Springer Verlag

\bibitem{PP}
W. Parry and M. Pollicott, Zeta functions and the periodic orbit strucuture of hyperbolic dynamics. Asterisque, Vol 187--188, pp 1--268 (1990).


\bibitem{PoSh}
M. Pollicott and R. Sharp,
Large Deviations, Fluctuations and Shrinking Intervals,
Commun. Math. Phys. 290, 321–-334 (2009)

\bibitem{Rue}
D. Ruelle, A generalized detailed balance relation, J. Stat. Phys., 164, 463–-471 (2016)

\bibitem{Rue1}
D. Ruelle, Positivity of entropy production in nonequilibrium statistical mechanics, Journal of Statistical Physics, Vol. 85, 1--25 (1996)

\bibitem{SV} R. R. Souza and V. Vargas,
Existence of Gibbs states and maximizing measures on a general one-dimensional lattice system with markovian structure, Qual. Theory Dyn. Syst. 21 (1), Art: 5  (2022)

\bibitem{VV1} V. Vargas,   On involution kernels and large deviations principles on $\beta$-shifts,
Discrete Contin. Dyn. Syst. 42 (6): 2699--2718 (2022)


\bibitem{Wal1}
P.~Walters.
\newblock {\em An introduction to ergodic theory}, volume~79 of {\em Graduate
  Texts in Mathematics}.
\newblock Springer-Verlag, (1982)

\bibitem{Wal2}
P.~Walters.
\newblock A natural space of functions for the {R}uelle operator theorem.
\newblock {\em Ergodic Theory Dynam. Systems}, 27(4):1323--1348 (2007)

\end{thebibliography}
\end{document}